\documentclass{amsart}
\usepackage{setspace}
\usepackage{a4}
\usepackage{amsthm}
\usepackage{latexsym}
\usepackage{amsfonts}
\usepackage{graphicx}
\usepackage{textcomp}
\usepackage{cite}
\usepackage{enumerate}
\usepackage{amssymb}
\usepackage{hyperref}
\usepackage{amsmath}
\usepackage{tikz}
\usepackage[mathscr]{euscript}
\usepackage{mathtools}
\newtheorem{theorem}{Theorem}[section]

\newtheorem{definition}[theorem]{Definition}

\newtheorem{lemma} [theorem]{Lemma}

\newtheorem{problem}[theorem]{Problem}

\newtheorem{question}[theorem]{Question}
\title{This is the title}
\usepackage{fancyhdr}

\begin{document}
\hrule\hrule\hrule\hrule\hrule
\vspace{0.3cm}	
\begin{center}
{\bf{Product Entropic Uncertainty Principle}}\\
\vspace{0.3cm}
\hrule\hrule\hrule\hrule\hrule
\vspace{0.3cm}
\textbf{K. Mahesh Krishna}\\
School of Mathematics and Natural Sciences\\
Chanakya University Global Campus\\
NH-648, Haraluru Village\\ 
Devanahalli Taluk, 	Bengaluru  Rural District\\
Karnataka  562 110, India\\
Email: kmaheshak@gmail.com\\

Date: \today
\end{center}

\hrule\hrule
\vspace{0.5cm}
\textbf{Abstract}: Motivated from Deutsch entropic uncertainty principle and several product uncertainty principles, we derive an uncertainty principle for the product of entropies using  functions.

\textbf{Keywords}:   Uncertainty Principle,  Frame.

\textbf{Mathematics Subject Classification (2020)}: 42C15.\\

\hrule

\hrule
\section{Introduction}
Let $\mathcal{H}$ be a finite dimensional Hilbert space. Given an orthonormal basis  $\{\omega_j\}_{j=1}^n$ for $\mathcal{H}$, the \textbf{Shannon entropy}  at a point $h \in \mathcal{H}_\tau$ is defined as 
\begin{align*}
	S_\tau (h)\coloneqq  \sum_{j=1}^{n} \left|\left \langle h, \tau_j\right\rangle \right|^2\log \left(\frac{1}{\left|\left \langle h, \tau_j\right\rangle \right|^2}\right)\geq 0, 
\end{align*}
where $\mathcal{H}_\tau \coloneqq \{h \in \mathcal{H}:\|h\|=1,  \langle h , \tau_j \rangle\neq 0, 1\leq j \leq n \}$. In 1983, Deutsch derived following uncertainty principle for Shannon entropy  \cite{DEUTSCH}.
\begin{theorem} (\textbf{Deutsch Uncertainty Principle}) \cite{DEUTSCH} \label{DU}
Let $\{\tau_j\}_{j=1}^n$,  $\{\omega_j\}_{j=1}^n$ be two orthonormal bases for a  finite dimensional Hilbert space $\mathcal{H}$. Then 
	\begin{align}\label{DUP}
2 \log n \geq S_\tau (h)+S_\omega (h)\geq -2 \log \left(\frac{1+\displaystyle \max_{1\leq j, k \leq n}|\langle\tau_j , \omega_k\rangle|}{2}\right)	\geq 0, \quad \forall h \in \mathcal{H}_\tau \cap \mathcal{H}_\omega.
	\end{align}
\end{theorem}
Using Buzano inequality, it is easy to see that Theorem \ref{DU} holds for Parseval frames \cite{KRISHNA}. Note that  Inequality (\ref{DUP}) is for the sum of entropies and not for the product. It is best if we have uncertainty principles for the product because  once we have an uncertainty principle for the product, then the uncertainty principle for the sum follows from the AM-GM inequality. Following is a small list of uncertainty principle (UP) for the product.
\begin{enumerate}
	\item Heisenberg-Robertson-Schrodinger UP and its various generalizations \cite{SCHRODINGER, VONNEUMANNBOOK, ROBERTSON, HEISENBERG, FOLLANDSITARAM, SELIG, GOHMICCHELLI}.
	\item Donoho-Stark-Elad-Bruckstein-Ricaud-Torrésani support size UP \cite{ELADBRUCKSTEIN, RICAUDTORRESANI, DONOHOSTARK}.
	\item Smith UP and its generalizations \cite{SMITH, ALAGICRUSSELL, CHUANGNG, MESHULAM}.
	\item Kuppinger-Durisi-Bölcskei support size UP and its generalization \cite{KUPPINGER, STUDER}.
	\item A. Wigderson-Y. Wigderson UP \cite{WIGDERSON}.
	\item Maccone-Pati UP \cite{MACCONEPATI}.
	\item Goh-Goodman UP \cite{GOHGOODMAN}.
	\item Jiang-Liu-Wu subfactor UP \cite{JIANGLIUWU}.
\item Bandeira-Lewis-Mixon	numerical sparsity UP \cite{BANDEIRALEWISMIXON}.
	\item UP over finite fields \cite{EVRA, BORELLO, FENG}.
	\item Generalized UP \cite{BOSSOLUCIANO, TAWFIKDIAB, KEMPFMANGANOMANN}.
\end{enumerate}
	It seems that for logarithm functions, we can not have uncertainty for the product of entropies using coherence. In this paper,  we derive an uncertainty principle for  the product of entropies coming from certain functions.

\section{Product Entropic Uncertainty Principle}
In the paper,   $\mathbb{K}$ denotes $\mathbb{C}$ or $\mathbb{R}$ and $\mathcal{H}$  denotes a Hilbert space  (need not be finite dimensional) over $\mathbb{K}$.  We need the notion of continuous frame which  is  introduced independently by Ali, Antoine and Gazeau \cite{ALIANTOINEGAZEAU} and Kaiser \cite{KAISER}. 
\begin{definition}\cite{ALIANTOINEGAZEAU, KAISER}
	Let 	$(\Omega, \mu)$ be a measure space. A collection   $\{\tau_\alpha\}_{\alpha\in \Omega}$ in 	a  Hilbert  space $\mathcal{H}$ is said to be a \textbf{continuous Parseval frame}  for $\mathcal{H}$ if the following conditions hold.
	\begin{enumerate}[\upshape(i)]
		\item For each $h \in \mathcal{H}$, the map $\Omega \ni \alpha \mapsto \langle h, \tau_\alpha \rangle \in \mathbb{K}$ is measurable.
		\item 
		\begin{align*}
			\|h\|^2=\int\limits_{\Omega}|\langle h, \tau_\alpha \rangle|^2\,d\mu(\alpha), \quad \forall h \in \mathcal{H}.
		\end{align*}
	\end{enumerate}
\end{definition}
Recall that a continuous Parseval frame $\{\tau_\alpha\}_{\alpha\in \Omega}$   for $\mathcal{H}$ 	is said to be \textbf{1-bounded} if 
\begin{align*}
\|\tau_\alpha\|\leq 1, \quad \forall \alpha \in \Omega.
\end{align*}
Let $\phi:(0,1]\to (0, \infty)$ be a  function satisfying following.
\begin{enumerate}
	\item $\phi$ is continuous.
	\item $\phi$ is decreasing.
	\item $\phi(xy)\leq \phi(x)\phi(y)$, $ \forall x, y \in (0,1] $.
\end{enumerate}
In the entire paper, we assume that $\phi$ satisfies above conditions. Given such a $\phi$ and a 1-bounded continuous Parseval frame $\{\tau_\alpha\}_{\alpha\in \Omega}$   for $\mathcal{H}$, we define entropy 
	\begin{align*}
	S_{\tau,\phi}(h)=\int\limits_{\Omega}|\langle h, \tau_\alpha \rangle|^2\phi(|\langle h, \tau_\alpha \rangle|^2)\,d\mu(\alpha)\geq 0, \quad \forall h \in \mathcal{H}_\tau, 
\end{align*}
where $\mathcal{H}_\tau \coloneqq \{h \in \mathcal{H}:\|h\|=1,  \langle h , \tau_\alpha \rangle\neq 0, \forall \alpha \in \Omega \}$. Since $\phi$ is continuous, the integral exists. Before deriving the theorem   of this paper, we need a powerful inequality.
\begin{lemma} (\textbf{Buzano Inequality}) \cite{BUZANO, FFUJIIKUBO, STEELE} \label{BI}
	Let $\mathcal{H}$ be a Hilbert space. Then  
	\begin{align*}
		|\langle u, h \rangle \langle  h,v \rangle|\leq \|h\|^2\frac{\|u\|\|v\|+|\langle  u,v \rangle|}{2}, \quad \forall h, u, v \in \mathcal{H}.
	\end{align*}
\end{lemma}
\begin{theorem}\label{PD}
Let $(\Omega, \mu)$,  $(\Delta, \nu)$ be   measure spaces. Let $\{\tau_\alpha\}_{\alpha\in \Omega}$ and $\{\omega_\beta\}_{\beta\in \Delta}$ be 1-bounded 
	continuous Parseval frames   for $\mathcal{H}$. Then 
	\begin{align}\label{PDI}
	S_{\tau,\phi}(h)S_{\omega,\phi}(h)\geq 	\phi\left(\frac{\left(1+\sup_{\alpha\in \Omega, \beta\in \Delta}|\langle \tau_\alpha, \omega_\beta \rangle|\right)^2}{4}\right)\geq 0, \quad \forall h \in \mathcal{H}_\tau \cap \mathcal{H}_\omega.
	\end{align}
\end{theorem}
\begin{proof}
Let $h \in \mathcal{H}_\tau \cap \mathcal{H}_\omega$. Then using the properties of $\phi$ and Lemma 	\ref{BI}, we get 
	\begin{align*}
		S_{\tau,\phi}(h)S_{\omega,\phi}(h)&=\left(\int\limits_{\Omega}|\langle h, \tau_\alpha \rangle|^2\phi(|\langle h, \tau_\alpha \rangle|^2)\,d\mu(\alpha)\right)\left(\int\limits_{\Delta}|\langle h, \omega_\beta \rangle|^2\phi(|\langle h, \omega_\beta \rangle|^2)\,d\nu(\beta)\right)\\
		&=\int\limits_{\Omega}\int\limits_{\Delta}|\langle h, \tau_\alpha \rangle|^2|\langle h, \omega_\beta \rangle|^2\phi(|\langle h, \tau_\alpha \rangle|^2)\phi(|\langle h, \omega_\beta \rangle|^2)\,d\nu(\beta)\,d\mu(\alpha)\\
		&\geq \int\limits_{\Omega}\int\limits_{\Delta}|\langle h, \tau_\alpha \rangle|^2|\langle h, \omega_\beta \rangle|^2\phi(|\langle h, \tau_\alpha \rangle|^2|\langle h, \omega_\beta \rangle|^2)\,d\nu(\beta)\,d\mu(\alpha)\\
		&\geq \int\limits_{\Omega}\int\limits_{\Delta}|\langle h, \tau_\alpha \rangle|^2|\langle h, \omega_\beta \rangle|^2\phi(|\langle h, \tau_\alpha \rangle|^2|\langle h, \omega_\beta \rangle|^2)\,d\nu(\beta)\,d\mu(\alpha)\\
		&\geq \int\limits_{\Omega}\int\limits_{\Delta}|\langle h, \tau_\alpha \rangle|^2|\langle h, \omega_\beta \rangle|^2\phi\left(\frac{\left(\|\tau_\alpha\|\omega_\beta\|+|\langle \tau_\alpha, \omega_\beta \rangle|\right)^2}{4}\right)\,d\nu(\beta)\,d\mu(\alpha)\\
		&\geq \int\limits_{\Omega}\int\limits_{\Delta}|\langle h, \tau_\alpha \rangle|^2|\langle h, \omega_\beta \rangle|^2\phi\left(\frac{\left(1+\sup_{\alpha\in \Omega, \beta\in \Delta}|\langle \tau_\alpha, \omega_\beta \rangle|\right)^2}{4}\right)\,d\nu(\beta)\,d\mu(\alpha)\\
		&=\phi\left(\frac{\left(1+\sup_{\alpha\in \Omega, \beta\in \Delta}|\langle \tau_\alpha, \omega_\beta \rangle|\right)^2}{4}\right)\int\limits_{\Omega}\int\limits_{\Delta}|\langle h, \tau_\alpha \rangle|^2|\langle h, \omega_\beta \rangle|^2\,d\nu(\beta)\,d\mu(\alpha)\\
		&=\phi\left(\frac{\left(1+\sup_{\alpha\in \Omega, \beta\in \Delta}|\langle \tau_\alpha, \omega_\beta \rangle|\right)^2}{4}\right).
	\end{align*}
\end{proof}
Theorem  \ref{PD}   gives the following question.
\begin{question}
	 Let $(\Omega, \mu)$,  $(\Delta, \nu)$ be   measure spaces,    $\mathcal{H}$ be a Hilbert space. For which pairs of 1-bounded continuous Parseval  frames $\{\tau_\alpha\}_{\alpha\in \Omega} $  and   $\{\omega_\beta\}_{\beta\in \Delta}$ for  $\mathcal{H}$, we have equality in Inequality (\ref{PDI})?
\end{question}
In 1988, Maassen  and Uffink (motivated from the conjecture  by Kraus made in 1987 \cite{KRAUS}) improved Deutsch entropic uncertainty principle.
\begin{theorem}\cite{MAASSENUFFINK} 
	(\textbf{Maassen-Uffink Entropic Uncertainty Principle})  \label{MU}
	Let $\{\tau_j\}_{j=1}^n$,  $\{\omega_k\}_{k=1}^n$ be two orthonormal bases for a  finite dimensional Hilbert space $\mathcal{H}$. Then 
	\begin{align*}
		S_\tau (h)+S_\omega (h)\geq -2 \log \left(\displaystyle \max_{1\leq j, k \leq n}|\langle\tau_j , \omega_k\rangle|\right), \quad \forall h \in \mathcal{H}_\tau \cap \mathcal{H}_\omega.
	\end{align*}	
\end{theorem}
In 2013, Ricaud  and Torr\'{e}sani \cite{RICAUDTORRESANI} showed that orthonormal bases in Theorem \ref{MU} can be improved to Parseval frames. 
\begin{theorem}\cite{RICAUDTORRESANI} 
	(\textbf{Ricaud-Torr\'{e}sani Entropic Uncertainty Principle})  \label{RT}
	Let $\{\tau_j\}_{j=1}^n$,  $\{\omega_k\}_{k=1}^m$ be two Parseval frames  for a  finite dimensional Hilbert space $\mathcal{H}$. Then 
	\begin{align*}
		S_\tau (h)+S_\omega (h)\geq -2 \log \left(\displaystyle \max_{1\leq j \leq n, 1\leq k \leq m}|\langle\tau_j , \omega_k\rangle|\right), \quad \forall h \in \mathcal{H}_\tau \cap \mathcal{H}_\omega.
	\end{align*}	
\end{theorem}
Proofs of  Theorems \ref{MU} and  \ref{RT} use Riesz-Thorin interpolation  and the differentiability   of logarithm function. We therefore  end by formulating the following problem. 
\begin{problem}
	Let $\{\tau_\alpha\}_{\alpha\in \Omega}$ and $\{\omega_\beta\}_{\beta\in \Delta}$ be 1-bounded 
	continuous Parseval frames   for $\mathcal{H}$. For which functions $\phi$, we have 
	\begin{align*}
		S_{\tau,\phi} (h)S_{\omega, \phi} (h)\geq   \phi\left(\left(\sup_{\alpha\in \Omega, \beta\in \Delta}|\langle \tau_\alpha, \omega_\beta \rangle|\right)^2\right), \quad \forall h \in \mathcal{H}_\tau \cap \mathcal{H}_\omega.
	\end{align*}		
\end{problem}

 \bibliographystyle{plain}
 \bibliography{reference.bib}

\end{document}